\documentclass[11pt]{amsart}
\usepackage{amsmath, amssymb, comment, graphicx, url, amsthm, multicol, latexsym, enumerate, bbm, mathtools, physics, microtype, cite, xcolor, float}
\usepackage{graphicx}
\usepackage{units}
\usepackage[all,cmtip]{xy}
\usepackage{wrapfig}
\usepackage[hmargin=1in]{geometry} 
\usepackage{tikz}
\usepackage{tikzsymbols}
\usetikzlibrary{shapes.geometric, arrows}
\usepackage[cm]{sfmath}
\usepackage{hyperref}
\usepackage{indentfirst}
\newcommand{\sym}[1]{\mathfrak{S}_{#1}}
\usepackage[sets,vectors,LinAlg,misc,polytopes]{shortmath}
\usepackage{luatex85}
\usepackage{tikz-3dplot}

\newcommand{\parallelogram}{
    \centering
    \begin{tikzpicture}[scale=0.1]
        \draw (0,0) -- (2,0) -- (3,2) -- (1,2) -- (0,0);
    \end{tikzpicture}
}

\hypersetup{
    colorlinks=true,
    linkcolor=blue,
    citecolor=magenta,
    filecolor=magenta,      
    urlcolor=magenta,
}

\newtheorem{theorem}{Theorem}[section]

\newtheorem{proposition}[theorem]{Proposition}
\newtheorem{lemma}[theorem]{Lemma}
\newtheorem{corollary}[theorem]{Corollary}

\theoremstyle{definition}
\newtheorem{remark}[theorem]{Remark}

\newtheorem{example}[theorem]{Example}

\newtheorem{definition}[theorem]{Definition}
\theoremstyle{remark}

\newcommand{\defterm}[1]{\emph{#1}}

\newcommand{\ba}{\mathbf{a}}
\newcommand{\bp}{\mathbf{p}}

\newcommand{\bx}{\mathbf{x}}

\newcommand{\be}{\mathbf{e}}

\newcommand{\bpi}{\boldsymbol{\pi}}

\newcommand{\EVol}{\mathrm{EVol}}

\DeclareMathOperator{\conv}{conv}

\makeatletter
\newtheorem*{rep@theorem}{\rep@title}\newcommand{\newreptheorem}[2]{
\newenvironment{rep#1}[1]{
\def\rep@title{\bf #2 \ref{##1}}
\begin{rep@theorem}}
{\end{rep@theorem}}}
\makeatother
\newreptheorem{theorem}{Theorem}

\newtheorem*{rep@proposition}{\rep@title}\newcommand{\newrepproposition}[2]{
\newenvironment{rep#1}[1]{
\def\rep@title{\bf #2 \ref{##1}}
\begin{rep@proposition}}
{\end{rep@proposition}}}
\makeatother
\newreptheorem{proposition}{Proposition}

\newtheorem*{rep@corollary}
{\rep@title}\newcommand{\newrepcorollary}[3]{
\newenvironment{rep#1}[1]{
\def\rep@title{\bf #2 \ref{##1}}
\begin{rep@corollary}}
{\end{rep@corollary}}}
\makeatother
\newreptheorem{corollary}{Corollary}

\definecolor{andresblue}{rgb}{0,0.72,0.92}
\definecolor{andrespink}{rgb}{1,0,1}

\begin{document}

\title{On the Geometry of Stack-Sorting Simplices}

\author{Cameron Ake}
\address{\scriptsize{Department of Mathematics, Harvey Mudd College}}
\email{\scriptsize{cake@g.hmc.edu}}

\author{Spencer F. Lewis}
\address{\scriptsize{Department of Mathematics, Harvey Mudd College}}
\email{\scriptsize{splewis@g.hmc.edu}}

\author{Amanda Louie}
\address{\scriptsize{Department of Mathematics, Harvey Mudd College}}
\email{\scriptsize{amlouie@g.hmc.edu}}

\author{Andr\'es R. Vindas-Mel\'endez}
\address{\scriptsize{Department of Mathematics, Harvey Mudd College}, \url{https://math.hmc.edu/arvm/}}
\email{\scriptsize{avindasmelendez@g.hmc.edu}}

\begin{abstract}
We show that all stack-sorting polytopes are simplices.
Furthermore, we show that the stack-sorting polytopes generated from $Ln1$ permutations have relative volume 1.
We establish an upper bound for the number of lattice points in a stack-sorting polytope.
In particular, stack-sorting polytopes generated from $2Ln1$ permutations have no interior points.
\end{abstract}

\maketitle

\section{Introduction}
A stack is a data structure that only allows the most recently input element to be output. 
Donald Knuth first popularized the idea of sorting a permutation by storing its values on a stack \cite{knuth68}.
Twenty-two years later, Julian West introduced a deterministic approach: the stack-sorting algorithm \cite{west90}.
One pass of the stack-sorting algorithm has time complexity O(n), but the algorithm is not guaranteed to fully sort a permutation in ascending order.
So, after we pass a permutation through the algorithm, we pass the sorted result through the algorithm again.
If we repeat this, the permutation will be sorted in ascending order after at most $n - 1$ iterations of the algorithm.
We will discuss the stack-sorting algorithm in depth in Section \ref{sec:background}.

In \cite{lee2025stacksortingsimplicesgeometrylatticepoint}, Lee, Mitchell, and Vindas-Mel\'endez analyzed the geometry of \textit{stack-sorting polytopes}.
To generate a stack-sorting polytope from a permutation of numbers $1$ through $n$, we repeatedly pass the permutation through the stack-sorting algorithm.
Then, we consider these input and output permutations as a set of points.
The convex hull of these points will be a stack-sorting polytope, a subset of the permutahedron in $\mathbb{R}^n$.

In this paper, we continue the study of stack-sorting polytopes.
We provide some computational results on the stack-sorting algorithm.
Then, we focus on the geometry and volumes of stack-sorting polytopes.
We also discuss the interior lattice points of stack-sorting polytopes.
We conclude by utilizing triangulations to study the lattice points of these polytopes.

In Section \ref{sec:background}, we give an overview of the stack-sorting algorithm, convex lattice polytopes, and stack-sorting polytopes.
In Section \ref{sec:simplices}, we prove that all stack-sorting polytopes are simplices, which we state in Theorem \ref{thm:all-simplices}.

\begin{reptheorem}{thm:all-simplices}
Let $\bpi \in \mathfrak{S}_n$ be exactly $k$-stack-sortable.
Then, $\triangle := conv(\mathcal{S}^{\bpi})$ is a $k$-simplex.
\end{reptheorem}

In Section \ref{sec:volume-of-polytopes}, we give the volume of all $Ln1$ simplices.

\begin{reptheorem}{thm:relvol-1}
    Let $\bpi$ be an $Ln1$ permutation.
    The relative volume of $\triangle = \conv(\mathcal{S}^{\bpi})$ is $1$.
\end{reptheorem}

In Section \ref{sec:lattice}, we study the lattice points of stack-sorting polytopes.
We give an upper bound for the number of lattice points in a stack-sorting polytope generated by an $Ln1$ permutation.
We briefly discuss triangulations of stack-sorting polytopes.
We also show that stack-sorting polytopes generated by $2Ln1$ permutations are hollow in Proposition \ref{prop:2Ln1-hollow}.

\begin{repproposition}{prop:2Ln1-hollow}
Let $\bpi \in \sym{n}$ be of the form $2Ln1$, where $L$ is any permutation of $\{3, 4, \dots, (n-1)\}$.
The simplex $\triangle := conv(\mathcal{S}^{\bpi})$ is hollow.
In particular, any non-vertex lattice point of $\triangle$ lies on the facet formed by the convex hull of $\mathcal{S}^{\bpi} \setminus \{\mathbf{e}\}$.
\end{repproposition}

In Section \ref{sec:questions}, we conclude by giving questions, conjectures, and areas of future research.

\section{Background and Preliminaries} \label{sec:background}

In this section, we provide  some necessary background for studying stack-sorting polytopes.
We begin with the stack-sorting algorithm.
This algorithm has been widely studied in combinatorics, computer science, probability, and polyehdral geometry; for example, see \cite{Bona, Branden, Defant1, Defant2, DefantEngenMiller, Defant3, Defant4, lee2025stacksortingsimplicesgeometrylatticepoint, MasudaThomasTonksVallete}.

\subsection{Stack-sorting algorithm} \label{subsec:stack-sort-background} \text{}

The stack-sorting algorithm takes a permutation $\bpi \in \mathfrak{S}_n$ of the numbers 1 through $n$ as input.
The purpose of the algorithm is to sort these numbers in increasing order.
We denote the identity permutation by $\ve = 123\cdots n$.

Before we discuss the stack-sorting algorithm, we must define a stack.
A stack is commonly referred to as a ``last in, first out" data structure.
A stack outputs its values in the opposite order to which they are input.

We refer to inputting a value into a stack as a \textit{push}, and we refer to outputting a value from the stack as a \textit{pop}.
We call the most-recently-input element of the stack the \textit{top} element, $t$.

Now, we define the stack-sorting algorithm.

\begin{definition} \label{def:stack-sorting-alg}
The \defterm{stack-sorting algorithm} on a permutation $\bpi$ is defined in West's thesis \cite{west90}.
We start with an empty stack and a permutation $\bpi = \pi_1 \pi_2 \cdots \pi_n$, and we iterate through the elements of $\bpi$, from left to right.
For each element $\pi_k$:
\begin{itemize}

    \item We compare $\pi_k$ to the value at the top of the stack, $t$.
    
    \item While $\pi_k$ is greater than $t$, we pop $t$ from the stack into our output.
    We repeat this step until $\pi_k$ is less than $t$, or until the stack is empty.

    \item Once $\pi_k$ is less than $t$, or once the stack is empty, we push $\pi_k$ onto the stack and repeat the process with $\pi_{k + 1}$.
    
\end{itemize}
Finally, once we have iterated through all elements of $\bpi$, we pop the remaining values on the stack.
We call the resulting permutation $s(\bpi)$.

It is important to note that all elements of $\bpi$ will be pushed onto the stack and popped into the output. 
Also note that the stack is always strictly decreasing.
\end{definition}

We illustrate this process with an example.

\begin{example} \label{ex:231_stacksort}
We pass $\bpi = 231$ through the stack-sorting algorithm. 
We begin by pushing the first element of $\bpi$, $2$, onto the stack.

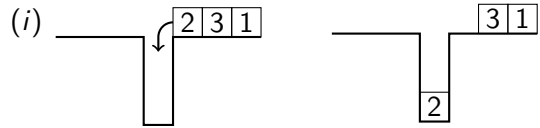
\begin{figure}[!ht]
    \centering
    \begin{tikzpicture}[scale = .39]
        \node at (-1, 0.5) {$(i)$};
        \draw[thick] (0,0)--(3,0)--(3,-3)--(4,-3)--(4,0)--(7,0);
        \draw (4,0) rectangle node {$2$} (5,1);
        \draw (5,0) rectangle node {$3$} (6,1);
        \draw (6,0) rectangle node {$1$} (7,1);
        \draw[thick, ->] (4,0.5) arc (90:180:0.5cm) -- (3.5, -.5);
    \end{tikzpicture}
    \qquad
    \begin{tikzpicture}[scale=0.39]
        \draw[thick] (0,0)--(3,0)--(3,-3)--(4,-3)--(4,0)--(7,0);
        \draw (3,-3) rectangle node {$2$} (4,-2);
        \draw (5,0) rectangle node {$3$} (6,1);
        \draw (6,0) rectangle node {$1$} (7,1);
        \end{tikzpicture}
\end{figure}

Now, we consider the next element of $\bpi$, $3$.
We see $3$ is greater than the top element of the stack, $2$. So, we pop $2$.

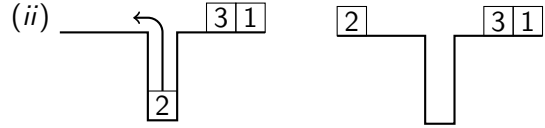
\begin{figure}[!ht]
    \centering
    \begin{tikzpicture}[scale=0.39]
        \node at (-1, 0.5) {$(ii)$};
        \draw[thick] (0,0)--(3,0)--(3,-3)--(4,-3)--(4,0)--(7,0);
        \draw (3,-3) rectangle node {$2$} (4,-2);
        \draw (5,0) rectangle node {$3$} (6,1);
        \draw (6,0) rectangle node {$1$} (7,1);
        \draw[thick, ->] (3.5,-2) -- (3.5, 0) arc (0:90:.5cm) -- (2.5, .5);
    \end{tikzpicture}
    \qquad
    \begin{tikzpicture}[scale=0.39]
        \draw[thick] (0,0)--(3,0)--(3,-3)--(4,-3)--(4,0)--(7,0);
        \draw (0,0) rectangle node {$2$} (1,1);
        \draw (5,0) rectangle node {$3$} (6,1);
        \draw (6,0) rectangle node {$1$} (7,1);
        \end{tikzpicture}
\end{figure}

The stack is empty, so we push $3$.

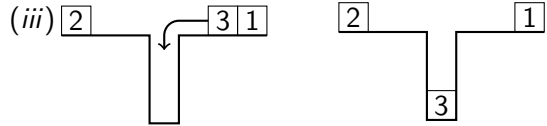
\begin{figure}[!ht]
    \centering
    \begin{tikzpicture}[scale = .39]
    \node at (-1, 0.5) {$(iii)$};
    \draw[thick] (0, 0) -- (3, 0) -- (3, -3) -- (4, -3) -- (4, 0) -- (7, 0);
    \draw (0, 0) rectangle node {$2$} (1, 1);
    \draw (5, 0) rectangle node {$3$} (6, 1);
    \draw (6, 0) rectangle node {$1$} (7, 1);
    \draw[thick, ->] (5, 0.5) -- (4, 0.5) arc (90:180:.5cm) -- (3.5, -.5);
    \end{tikzpicture}
    \qquad
    \begin{tikzpicture}[scale = .39]
    \draw[thick] (0, 0) -- (3, 0) -- (3, -3) -- (4, -3) -- (4, 0) -- (7, 0);
    \draw (0, 0) rectangle node {$2$} (1, 1);
    \draw (3, -3) rectangle node {$3$} (4, -2);
    \draw (6, 0) rectangle node {$1$} (7, 1);
    \end{tikzpicture}
\end{figure}

Now, we consider the next (and final) element of $\bpi$, $1$.
We see $1$ is less than the top element of the stack, $3$. So, we push $1$.

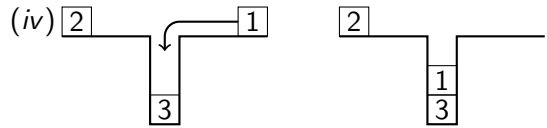
\begin{figure}[!ht]
    \centering
    \begin{tikzpicture}[scale = .39]
    \node at (-1, 0.5) {$(iv)$};
    \draw[thick] (0, 0) -- (3, 0) -- (3, -3) -- (4, -3) -- (4, 0) -- (7, 0);
    \draw (0, 0) rectangle node {$2$} (1, 1);
    \draw (3, -3) rectangle node {$3$} (4, -2);
    \draw (6, 0) rectangle node {$1$} (7, 1);
    \draw[thick, ->] (6, 0.5) -- (4, 0.5) arc (90:180:.5cm) -- (3.5, -.5);
    \end{tikzpicture}
    \qquad
    \begin{tikzpicture}[scale = .39]
    \draw[thick] (0, 0) -- (3, 0) -- (3, -3) -- (4, -3) -- (4, 0) -- (7, 0);
    \draw (0, 0) rectangle node {$2$} (1, 1);
    \draw (3, -3) rectangle node {$3$} (4, -2);
    \draw (3, -2) rectangle node {$1$} (4, -1);
    \end{tikzpicture}
\end{figure}

Now that we have iterated through the permutation, we pop the values that are in the stack.

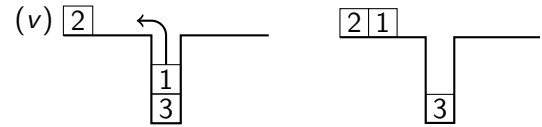
\begin{figure}[!ht]
    \centering
    \begin{tikzpicture}[scale = .39]
    \node at (-1, 0.5) {$(v)$};
    \draw[thick] (0, 0) -- (3, 0) -- (3, -3) -- (4, -3) -- (4, 0) -- (7, 0);
    \draw (0, 0) rectangle node {$2$} (1, 1);
    \draw (3, -3) rectangle node {$3$} (4, -2);
    \draw (3, -2) rectangle node {$1$} (4, -1);
    \draw[thick, ->] (3.5, -1) -- (3.5, 0) arc (0:90:.5cm) -- (2.5, 0.5);
    \end{tikzpicture}
    \qquad
    \begin{tikzpicture}[scale = .38]
    \draw[thick] (0, 0) -- (3, 0) -- (3, -3) -- (4, -3) -- (4, 0) -- (7, 0);
    \draw (0, 0) rectangle node {$2$} (1, 1);
    \draw (3, -3) rectangle node {$3$} (4, -2);
    \draw (1, 0) rectangle node {$1$} (2, 1);
    \end{tikzpicture}
\end{figure}
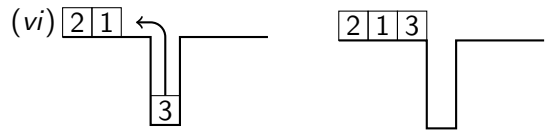
\begin{figure}[!ht]
    \centering
    \begin{tikzpicture}[scale = .39]
    \node at (-1, 0.5) {$(vi)$};
    \draw[thick] (0, 0) -- (3, 0) -- (3, -3) -- (4, -3) -- (4, 0) -- (7, 0);
    \draw (0, 0) rectangle node {$2$} (1, 1);
    \draw (1, 0) rectangle node {$1$} (2, 1);
    \draw (3, -3) rectangle node {$3$} (4, -2);
    \draw[thick, ->] (3.5, -2) -- (3.5, 0) arc (0:90:.5cm) -- (2.5, 0.5);
    \end{tikzpicture}
    \qquad
    \begin{tikzpicture}[scale = .39]
    \draw[thick] (0, 0) -- (3, 0) -- (3, -3) -- (4, -3) -- (4, 0) -- (7, 0);
    \draw (0, 0) rectangle node {$2$} (1, 1);
    \draw (1, 0) rectangle node {$1$} (2, 1);
    \draw (2, 0) rectangle node {$3$} (3, 1);
    \end{tikzpicture}
\end{figure}

This gives us $s(\bpi) = 213$.

\end{example}

As we see in Example \ref{ex:231_stacksort}, the stack-sorting algorithm does not always fully sort a permutation in ascending order.
We can pass $s(\bpi)$ through the algorithm again, yielding $s(s(\bpi))$, or $s^2(\bpi)$.
We continue applying the algorithm until our permutation is fully sorted.
A permutation of length $n$ will be sorted in ascending order after at most $n - 1$ iterations of the stack-sorting algorithm.

\begin{definition} \label{def:k-stack-sortable}
Given a permutation $\bpi \in \mathfrak{S}_n$, if $s^k(\bpi) = \mathbf{e}= 123\cdots n$, then $\bpi$ is \defterm{k-stack-sortable}.
If this is the smallest value $k$ for which $\bpi$ is $k$-stack-sortable, then $\bpi$ is \defterm{exactly $k$-stack-sortable.}
\end{definition}

Let $\mathcal{S}^{\bpi}$ be the set, including $\bpi$, of permutations resulting from subsequent passes of $\bpi$ through the stack-sorting algorithm.
That is, $\mathcal{S}^{\bpi} = \{\bpi, s(\bpi), s^2(\bpi), \dots, \ve\}$.
We illustrate $\mathcal{S}^{\bpi}$ with an example.
\begin{example} \label{ex:spi}
    Let $\bpi = 3241$.
    Then, $s(\bpi) = 2314$, $s^2(\bpi) = 2134$, and $s^3(\bpi) = 1234$.
    So, $\mathcal{S}^{\bpi} = \{3241, 2314, 2134, 1234\}$.
\end{example}
We study $\mathcal{S}^{\bpi}$ further when we discuss stack-sorting polytopes.

\subsection{Permutations ending in ascending or descending subsequences} \label{subsec:ending-sub-seq} \text{}

We now present results on the stack-sorting algorithm.
We start by exploring permutations ending in an ascending sequence.
We denote these permutations as $\bpi = AB$, where $A$ and $B$ are subsequences of $\bpi$ and $B$ is in ascending order.

\begin{proposition} \label{prop:ascent-ending}
Let $\bpi=AB\in \sym{n}$, where $B$ is some ascending sequence.
Then, $\bpi$ is $|A|$-stack-sortable.

\end{proposition}
\begin{proof}
    We proceed by strong induction on $|A|$.
    
    For the $|A| = 0$ case, $\bpi = \ve$, so $s^0(\bpi) = \bpi = \ve$.
 
    Now, assume that $s^{|A|}(\bpi) = \ve$ for $|A| < k$.
    We show that $s^{|A|}(\bpi) = \ve$ for $|A| = k$.

    Consider the stack-sorting algorithm on $\bpi$.
    First, the algorithm iterates through the values in $A$.
    When the algorithm considers the first element of $B$, there will be at least one value of $A$ on the stack.

    The stack is strictly decreasing, so the smallest value on the stack is its top value, $t$.
    Also, $B=b_1 b_2 \cdots b_r$ is strictly increasing, so the smallest value of $B$ is its first value, $b_1$.

    Now, consider two cases.
    If $b_1 > t$, then $t$ will be the next value popped.
    Otherwise, $b_1 < t$, and $b_1$ will be pushed.
    Then, the algorithm will consider $b_2$, which is greater than $b_1$.
    So, $b_1$ will be the next value popped.
    Either way, the next value popped is smallest of the remaining elements.

    Now, we again have a strictly decreasing stack and a strictly increasing sequence.
    So, the next value popped will be the next smallest, and so forth.
    In this manner, all of the remaining elements will be popped in ascending order, and $s(\bpi)$ ends in a strictly longer increasing sequence than $\bpi$.
    That is, $s(\bpi)$ can be written as $AB$, where $B$ is increasing and $|A| < k$.

    By the inductive hypothesis, $s(\bpi)$ is $(k - 1)$-stack-sortable.
    So, $\bpi$ is $k$-stack-sortable.
\end{proof}

We illustrate Proposition \ref{prop:ascent-ending} with an example.

\begin{example} \label{ex:2451367}
The permutation $\bpi = 2451367$ can be broken into $A = 245$ and $B = 1367$, where $|A| = 3$.
We have $s(\bpi) = 2413567$. 
Note how the value $5$ is inserted into the sequence $B = 1367$.
Next, we have $s^2(\bpi) = 2134567$. Here, $4$ is inserted into the sequence $13567$.
Finally, we have $s^3(\bpi) = 1234567$.
So, $\bpi = 2451367$ is $3$-stack-sortable.

\end{example}

\begin{remark} \label{rem:ascent-iff}
    Note that by a similar proof to Proposition \ref{prop:ascent-ending}, we have that $\bpi$ is exactly $|A|$-stack-sortable if and only if $A$ ends with its largest value, and $B$ begins with $1$.
\end{remark}

Next, we explore permutations ending in a descending sequence.
Again, we denote these permutations as $\bpi = AB$.
Now, $B$ is in descending order.

\begin{corollary} \label{cor:descent-ending}
Let $\bpi = AB \in \sym{n}$, where $B$ is a descending sequence.
Then, $\bpi$ is $(|A| + 1)$-stack-sortable.
\end{corollary}

\begin{proof}
    Consider the stack-sorting algorithm on $\bpi$.
    First, the algorithm iterates through the values in $A$.
    When the algorithm reaches the first element of $B$, $b_1$, this element will at some point be pushed onto the stack.
    Then, the other values will be pushed onto the stack on top of $b_1$, since $B$ is strictly decreasing.
    So, the output will end in the values of $B$ (and possibly some values of $A$) in increasing order.

    Now, after one iteration of the stack-sorting algorithm, $s(\bpi)$ can be written as $A'B'$, where $B'$ is some ascending sequence and $|A'| \leq |A|$.
    So, by Proposition \ref{prop:ascent-ending}, $s(\bpi)$ is $|A'|$-stack-sortable.
    So, $\bpi$ is $(|A| + 1)$-stack-sortable.
\end{proof}

We illustrate Corollary \ref{cor:descent-ending} with an example.

\begin{example} \label{ex:7541632}
The permutation $\bpi = 7541632$ can be decomposed into $A = 7541$ and $B = 632$, with $|A| = 4.$
We have $s(\bpi) = 1452367$, $s^2(\bpi) = 1423567$, and $s^3(\bpi) = 1234567$. 
Hence, $\bpi$ is 3-stack-sortable.
\end{example}

\begin{remark} \label{rem:descent-iff}
    Note that by a similar proof to Corollary \ref{cor:descent-ending}, and using Remark \ref{rem:ascent-iff}, $\bpi$ is exactly $(|A| + 1)$-stack-sortable if and only if $B$ starts with $n$ and ends with $1$.
\end{remark}

We will discuss stack-sorting algorithm and stack-sorting polytopes further in Subsections \ref{subsec:stack-sort-polytope-background} and \ref{subsec:Ln1}.

\subsection{Convex lattice polytopes} \label{Ehr-background} \text{}

This subsection provides background on convex lattice polytopes and their discrete geometry.
For more information on polytopes, see \cite{zieglerlecturesonpolytopes}.

The \defterm{convex hull} of a set of points $S$, $\conv(S)$, is the smallest convex set containing $S$.
A \defterm{convex polytope} $\mathcal{P}$ is the convex hull of a finite set $S = \{ \bx_1, \ldots, \bx_k \}$:
\begin{equation*}\label{eqn:conv-hull} 
\mathcal{P} =\conv(S) := \left\lbrace \sum_{i=1}^k \lambda_i \bx_i: \lambda_i \in \R_{\geq 0} \text{ where } \sum_{i=1}^k \lambda_i = 1 \right\rbrace.
\end{equation*}
\begin{figure}[ht]\label{ex:conv-hull}
    \begin{tikzpicture}
        [scale=1.5,
        vertex/.style={inner sep=1pt,circle,draw=andrespink,fill=andrespink,thick}]
        \draw[thick, ->] (-0.5,0) -- (1.5,0) node[anchor = north west] {};
        \draw[thick, ->] (0,-0.25) -- (0,1.25) node[anchor = south east] {};
        \coordinate (A) at (0, 0);
        \coordinate (B) at (1, 0);
        \coordinate (C) at (0, 1);

        \draw[black,top color=andresblue, bottom color=andresblue, fill opacity=0.3]  (C) -- (A) -- (B) -- cycle;

        \node[vertex] at (0, 0){};
        \node[vertex] at (1, 0){};
        \node[vertex] at (0, 1){};

        \node[anchor=north east] at (0, 0){\color{blue}$(0, 0)$};
        \node[anchor=north] at (1, 0){\color{blue}$(1, 0)$};
        \node[anchor=east] at (0, 1){\color{blue}$(0, 1)$};
    
    \end{tikzpicture}
    \caption{The convex hull of the set $\{(0, 0), (1, 0), (0, 1)\}$.}
\end{figure}
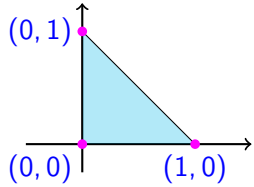

A \defterm{convex lattice polytope} is a convex polytope whose vertices are lattice points.
For the rest of the paper, all polytopes discussed will be convex lattice polytopes, unless otherwise specified.

The \defterm{affine span} of a polytope $\mathcal{P}$, $\mathrm{span}(\mathcal{P})$, is the smallest affine space that contains $\mathcal{P}$:
\[\mathrm{span}(\mathcal{P}) = \{\vx + \lambda(\vy - \vx) : \vx, \vy \in \mathcal{P}, \lambda \in \mathbb{R}\}.\]

The \defterm{dimension} of a polytope $\mathcal{P}$ is the dimension of $span(\mathcal{P})$.

We call a polytope with dimension $d$ a $d$-polytope.
If a polytope in $\mathbb{R}^n$ has dimension $n$, we refer to it as \defterm{full-dimensional}.
A $d$-polytope $\mathcal{P}$ must contain at least $d+1$ vertices.
If $\mathcal{P}$ contains exactly $d + 1$ vertices, we refer to $\mathcal{P}$ as a \defterm{simplex}.
For instance, the $2$-simplex is a triangle.

Now, we describe the volumes of polytopes.
We denote the Euclidean volume of polytope $\mathcal{P}$ as $\EVol(\mathcal{P})$.

\begin{definition} \label{def:vol}
    The \defterm{relative volume} of a polytope $\mathcal{P}$, $\vol(\mathcal{P})$, is the Euclidean volume of $\mathcal{P}$, normalized to $\mathrm{span}(\mathcal{P})$.
\end{definition}
We can compute the relative volume of $\mathcal{P}$ by computing its Euclidean volume, then dividing by the Euclidean volume of a fundamental parallelepiped.
A \defterm{fundamental parallelepiped} of an affine space is a parallelepiped whose generating vectors span the integer points of the space.
All fundamental parallelepipeds have the same volume.
For further explanation of the fundamental parallelepiped, see \cite{relvolreference}, which refers to the fundamental parallelepiped as the primitive parallelotope.

We use this discrete geometry throughout the paper to study a specific class of polytopes, which we introduce in Subsection \ref{subsec:stack-sort-polytope-background}.

\subsection{Stack-sorting polytopes} \label{subsec:stack-sort-polytope-background} \text{}

We now introduce the focus of our research, stack-sorting polytopes.
For the rest of this paper, we will make no distinction between permutations of $\mathfrak{S}_n$ and the corresponding points in $\mathbb{R}^n$.

\begin{definition} \label{def:stack-sorting-polytope}
    Given a permutation $\bpi,$ we define the \defterm{stack-sorting polytope} of $\bpi$ as the convex hull of $\mathcal{S}^{\bpi}$.
\end{definition}

\begin{example} \label{ex:231-polytope}
    Let $\bpi = 231$.
    We have $s(\bpi) = 213$ and $s^2(\bpi) = 123$. So, $\mathcal{S}^{\bpi} = \{231, 213, 123\}$, and the stack-sorting polytope of $\bpi$ is $\conv(\{231, 213, 123\})$.
    \begin{figure}[ht]
        \begin{tikzpicture}
            [scale=1,
            vertex/.style={inner sep=1pt,circle,draw=andrespink,fill=andrespink,thick}]
            \coordinate (A) at (2, 3, 1);
            \coordinate (B) at (2, 1, 3);
            \coordinate (C) at (1, 2, 3);

            \draw[black,top color=andresblue, bottom color=andresblue, fill opacity=0.3]  (C) -- (A) -- (B) -- cycle;

            \node[vertex] at (2, 3, 1){};
            \node[vertex] at (2, 1, 3){};
            \node[vertex] at (1, 2, 3){};

            \node[anchor=east] at (1, 2, 3) {\color{blue}$123$};
            \node[anchor=north] at (2, 1, 3) {\color{blue}$213$};
            \node[anchor=south west] at (2, 3, 1) {\color{blue}$231$};
        
        \end{tikzpicture}
        \caption{The simplex $\conv(\mathcal{S}^{\bpi})$ as described in Example \ref{ex:231-polytope}.}
    \end{figure}
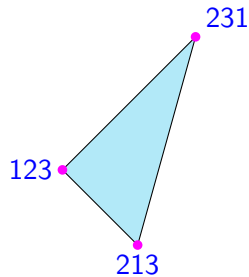
\end{example}

Note that all permutations $\bpi \in \mathfrak{S}_n$ lie on the $(n - 1)$-dimension hyperplane $x_1 + x_2 + \dots + x_n = 1 + 2 + \dots + n$.
So, we will use the concepts of Euclidean volume and relative volume to study these polytopes.

\subsection{Special permutations} \label{subsec:Ln1}

In this section, we discuss permutations of $\mathfrak{S}_n$ that end in $n1$.
The permutations are referred to as $Ln1$ permutaitons in \cite{lee2025stacksortingsimplicesgeometrylatticepoint}.

\begin{definition} \label{def:Ln1} 
    For $\bpi\in \mathfrak{S}_{n}$, $\bpi$ is an \defterm{$Ln1$} permutation if $\bpi$ ends with $n1$.
    Here, $L$ is some permutation of $\{2, 3, ..., (n-1)\}$ and
    $\mathcal{L}^n$ denotes the set of $Ln1$ permutations of length $n$. 
\end{definition}

Both Remark \ref{rem:ascent-iff} and Remark \ref{rem:descent-iff} are sufficient to prove Theorem 3.11 of \cite{lee2025stacksortingsimplicesgeometrylatticepoint}, which we restate here:
\begin{corollary}[Theorem 3.11,\cite{lee2025stacksortingsimplicesgeometrylatticepoint}] \label{cor:Ln1-(n-1)-stack-sortable}
    A permutation $\bpi \in \mathfrak{S}_n$ is exactly $(n-1)$-stack sortable if and only if $\bpi$ ends in $n1$.
\end{corollary}

We discuss $Ln1$ permutations further in Section \ref{sec:volume-of-polytopes}.

The authors of \cite{lee2025stacksortingsimplicesgeometrylatticepoint} also prove that all stack-sorting polytopes of $Ln1$ permutations are simplices.
In the next section, we generalize this result.

\section{Geometry of \texorpdfstring{$\conv\left(\mathcal{S}^{\bpi}\right)$}-} \label{sec:simplices}

In this section, we prove that any permutation $\bpi \in \mathfrak{S}_n$ guarantees that $\conv(\mathcal{S}^{\bpi})$ forms a simplex. 

\begin{definition} \label{def:fixed-values}
    Let $\bpi \in \mathfrak{S}_n$ be some permutation.
    We call a value $\pi_k$ of $\bpi$ a \defterm{fixed value} if $\pi_k = k$.
\end{definition}

\begin{lemma}\label{lemma:end-fixed-points}
    Suppose $\bpi \in \mathfrak{S}_n$ ends in exactly $k$ fixed values.
    Then, $s(\bpi)$ ends in at least $k + 1$ fixed values.
\end{lemma}
\begin{proof}
    All values of $\bpi$ greater than $n - k$ are fixed.
    So, the value $n - k$ is the largest of the first $n - k$ values in $\bpi$.

    We pass $\bpi$ through the stack-sorting algorithm.
    When the algorithm reaches the value $n - k$, we pop the remaining values off the stack and push $n - k$ onto the bottom of the stack. 
    
    When the algorithm reaches $\pi_{n - k + 1}$, $n - k$ will still be on the bottom of the stack.
    We pop the remaining values off the stack and $n - k$ is written in the $(n - k)^{\text{th}}$ position.
    The algorithm will then push and pop the remaining values of $\bpi$, which are in ascending order.

    So, $s(\bpi)$ ends in at least $k + 1$ fixed values.
\end{proof}

\begin{theorem} \label{thm:all-simplices}
Let $\bpi \in \mathfrak{S}_n$ be exactly $k$-stack-sortable.
Then, $\triangle := conv(\mathcal{S}^{\bpi})$ is a $k$-simplex.
\end{theorem}
\begin{proof}
    We show that $\triangle$ is a simplex by showing the vectors in $\mathcal{S}^{\bpi}$ are affinely independent.
    
    There are $k + 1$ vectors in $\mathcal{S}^{\bpi}$, with $s^{k} = \mathbf{e}$.
    Now, subtract $\ve$ component-wise from each vector in $\mathcal{S}^{\bpi} \setminus \{\ve\}$.
    We show these $k$ vectors are linearly independent.
    
    From Lemma \ref{lemma:end-fixed-points}, we know that $s^{i + 1}(\bpi)$ must end in a strictly greater number of fixed points than $s^i(\bpi)$.
    So, $s^{i + 1}(\bpi) - \ve$ ends in strictly more zeros than $s^{i}(\bpi)$.
    
    Now, consider the $k \times n$ matrix such that the $i^{\text{th}}$ row is the vector  $s^i(\bpi) - \ve$.
    The $i^{\textbf{th}}$ row ends in strictly more zeros than all rows above.
    
    Now, flip this matrix across a vertical axis, which is equivalent to relabeling our variables.
    Here, the $i^{\text{th}}$ row has strictly more leading zeros than the above rows.
    We also have no zero rows.
    So, our matrix is in Row Echelon Form, and our vectors $s^i(\bpi) - \ve$ are linearly independent.
    So, the $k + 1$ vectors $s^i(\bpi)$ are affinely independent.
    Therefore, $\triangle$ is a $k$-simplex.
\end{proof}

\begin{example} \label{ex:simplex}
    Let $\bpi = 31452$.
    We have $\mathcal{S}^{\bpi} = \{31452, 13425, 13245, 12345\}$.
    The convex hull of $\mathcal{S}^{\bpi}$ forms a $3$-simplex.
    \begin{figure}[ht] \label{fig:simplex}
        \begin{tikzpicture}
            [scale=1,
            vertex/.style={inner sep=1pt,circle,draw=andrespink,fill=andrespink,thick}]
            \coordinate (A) at (3, 1, 4, 5, 2);
            \coordinate (B) at (1, 3, 4, 2, 5);
            \coordinate (C) at (1, 3, 2, 4, 5);
            \coordinate (D) at (1, 2, 3, 4, 5);
    
            \draw[gray,top color=andresblue, bottom color=andresblue, fill opacity=0.3, dashed]  (A) -- (B) -- (C) -- cycle;
            \draw[gray,top color=andresblue, bottom color=andresblue, fill opacity=0.3]  (A) -- (C) -- (D) -- cycle;
            \draw[gray,top color=andresblue, bottom color=andresblue, fill opacity=0.3]  (B) -- (C) -- (D) -- cycle;
            
            \draw[gray,top color=andresblue, bottom color=andresblue, fill opacity=0.3, dashed] (A) -- (B);
            \draw[gray,top color=andresblue, bottom color=andresblue, fill opacity=0.3] (A) -- (C);
            \draw[gray,top color=andresblue, bottom color=andresblue, fill opacity=0.3] (A) -- (D);
            \draw[gray,top color=andresblue, bottom color=andresblue, fill opacity=0.3] (B) -- (C);
            \draw[gray,top color=andresblue, bottom color=andresblue, fill opacity=0.3] (B) -- (D);
            \draw[gray,top color=andresblue, bottom color=andresblue, fill opacity=0.3] (C) -- (D);
    
            \node[vertex] at (3, 1, 4, 5, 2){};
            \node[vertex] at (1, 3, 4, 2, 5){};
            \node[vertex] at (1, 3, 2, 4, 5){};
            \node[vertex] at (1, 2, 3, 4, 5){};
    
            \node[anchor=north west] at (3, 1, 4, 5, 2) {\color{blue}$31452$};
            \node[anchor=east] at (1, 3, 4, 2, 5) {\color{blue}$13425$};
            \node[anchor=south] at (1, 3, 2, 4, 5) {\color{blue}$13245$};
            \node[anchor = north east] at (1, 2, 3, 4, 5) {\color{blue}$12345$};
        
        \end{tikzpicture}
        \caption{The 3-simplex formed by $\mathcal{S}^{\bpi}$ in Example \ref{ex:simplex}.}
    \end{figure}
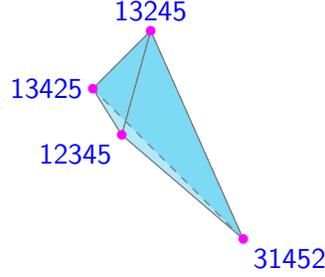
\end{example}

\section{Volume of \texorpdfstring{$\conv\left(\mathcal{S}^{\bpi}\right)$}-} \label{sec:volume-of-polytopes}

This section focuses on the relative and Euclidean volumes of $Ln1$ simplices. We prove the relative volume is 1 and the Euclidean volume is $\sqrt{n}$ for all $Ln1$ simplices.

\begin{lemma}\label{lemma:lattice-basis}
Let $\mathcal{B}$ be the set $\{(1, -1, 0, \dots, 0), (0, 1, -1, \dots, 0),  \dots, $ $ (0, \dots, 0, 1, -1)\}$ in $\mathbb{R}^n$, and let $H$ be the hyperplane $x_1 + x_2 + \dots + x_n = 1 + 2 + \dots + n$. Then, $\mathcal{B}$ forms a basis for the lattice $\mathbb{Z}^n \cap H$.
\end{lemma}

\begin{proof}
We show that $\mathcal{B}$ forms a basis for the lattice $\mathbb{Z}^n \cap H'$, where $H'$ is the hyperplane $x_1 + x_2 + ... + x_n = 0$, a translation of $H$.
Consider an arbitrary lattice point $\ba = (a_1, a_2, \dots, a_n)$ on $H'$.
Note that $\ba$ must satisfy $a_n = -(a_1 + a_2 + \dots + a_{n - 1})$.
We can take the following linear combination:
\begin{align*}
    a_1 \cdot (1, -1, 0, 0, \dots, 0) + (a_1 + a_2) \cdot (0, 1, -1, 0, \dots, 0)  + \dots &+ (a_1 + a_2 + \dots + a_{n - 1})\cdot (0, 0, \dots, 0, 1, -1)\\ &= (a_1, a_2, ..., a_{n - 1}, -(a_1 + a_2 + \dots + a_{n - 1})) \\&=
    (a_1, a_2, \dots, a_n).
\end{align*}
So, the vectors in $\mathcal{B}$ span the lattice.

Also, the vectors in $\mathcal{B}$ are linearly independent.
So, $\mathcal{B}$ forms a basis for $\mathbb{Z}^n \cap H'$.
Thus, $\mathcal{B}$ also forms a basis for $\mathbb{Z}^n \cap H$.
\end{proof}

Now, we prove that the Euclidean volume of the fundamental parallelepiped of an $Ln1$ simplex is $\sqrt{n}.$
\begin{lemma}\label{lemma:parallelepiped-vol}
Let $\bpi \in \mathcal{L}^n$.
The fundamental parallelepiped $\parallelogram$ for $\mathrm{span}(\conv(\mathcal{S}^{\bpi}))$ has Euclidean volume $\sqrt{n}.$
\end{lemma}

\begin{proof}
    Note that $\conv(\mathcal{S}^{\bpi})$ is an $(n - 1)$-dimension polytope living on the $(n-1)$-dimension hyperplane \[H := \{\vx \in \mathbb{R}^n \mid x_1 + x_2 + \dots + x_n = 1 + 2 + \dots + n\}.\]
    So, $\mathrm{span}(\conv(\mathcal{S}^{\bpi})) = H$.
    From Lemma \ref{lemma:lattice-basis}, \[\mathcal{B} = \{(1, -1, 0, 0, \dots, 0),  (0, 1, -1, 0, \dots, 0), (0, 0, 1, -1, \dots, 0),\dots, (0, \dots, 0, 1, -1)\}\] forms a basis for the lattice $\mathbb{Z}^n \cap H$.
    We can take these basis vectors as generators for a fundamental parallelepiped $\parallelogram.$
    
    Using Theorem 7 of \cite{peng2007determinant}, we calculate the Euclidean volume of $\parallelogram$ by $\EVol(\parallelogram)^2 = \det(AA^T)$, where $A$ is the $n - 1 \times n$ matrix whose rows are the vectors in $\mathcal{B}$.
    We evaluate:
    \begin{align*}
        AA^T = \mat{1 & -1 & 0 & \cdots & 0 \\
                  0 & 1 & -1 & \cdots & 0 \\
                  \vdots & \vdots & \vdots & \ddots & \vdots \\
                  0 & 0 & 0 & \cdots & -1}
             \mat{1 & 0 & \cdots & 0 \\
                  -1 & 1 & \cdots & 0 \\
                  0 & -1 & \cdots & 0 \\
                  \vdots & \vdots & \ddots & \vdots \\
                  0 & 0 & \cdots & -1}
        = \mat{2 & -1 & 0 & \cdots & 0 \\
            -1 & 2 & -1 & \cdots & 0 \\
            0 & -1 & 2 & \cdots & 0 \\
            \vdots & \vdots & \vdots & \ddots & \vdots \\
              0 & 0 & 0 & \cdots & 2}.
    \end{align*}
    Note that $AA^T$ is $(n - 1) \times (n - 1)$.
    Call this matrix $A_{n - 1}$, and call its determinant $d_{n - 1}$. 
    We compute $d_{n - 1}$ by cofactor expansion across the top row:
    \begin{align*}
    d_{n - 1}
            &= 2\det(A_{n - 2})
            - (-1) \det \left(
                \begin{bmatrix}
                    -1 & -1 & 0 & 0 & \dots & 0 \\
                    \textbf{0} & & & A_{n - 3}
                \end{bmatrix} \right). \\
    \end{align*}
    For this second matrix, we again expand across the top row:
    \begin{align*}
    d_{n - 1}
            &= 2 \det(A_{n - 2})
            - (-1)(-1) \det(A_{n - 3})\\
            &= 2d_{n-2} - d_{n-3}.
    \end{align*}
    
    Now, we show that $d_{n - 1} = n$ by induction.
    
    We have base cases of $n = 2$ and $n = 3$:
    \begin{align*}
        d_1 &= \det\left(\begin{bmatrix}2\end{bmatrix}\right) = 2\text{, and}\\
        d_2 &= \det\left(\begin{bmatrix} 2 & -1 \\ -1 & 2 \end{bmatrix}\right) = 3.
    \end{align*}
    
    For our inductive hypothesis, we assume $d_{n - 2} = n - 1$ and $d_{n - 3} = n - 2$.
    For our inductive step, we show $d_{n - 1} = n$.
    We have $d_{n - 1} = 2d_{n - 2} - d_{n - 3} = 2(n - 1) - (n - 2) = n$. 
    So, $\EVol(\parallelogram)^2 = n,$ and the Euclidean volume of $\parallelogram$ is $\sqrt{n}.$
\end{proof}

Now that we have established the Euclidean volume of the fundamental parallelepiped, we find the Euclidean volume of the simplex itself. 

\begin{lemma}\label{lemma:end-zeros}
Let $\bpi \in \mathcal{L}^n$.
Then, for $0 \leq k \leq n - 2$, the vector $s^k(\bpi) - \ve$ ends with \[((-n+k+1), 0, 0, \dots, 0).\]
Here, $s^k(\bpi) - \ve$ ends with exactly $k$ zeros.
\end{lemma}
\begin{proof}
    By Theorem 3.5 of \cite{lee2025stacksortingsimplicesgeometrylatticepoint},
    $s^k(\bpi)$ ends with $(1, n - k + 1, n - k + 2, \dots, n).$
    So, $s^k(\bpi) - \ve$ ends with $((-n + k + 1), 0, 0, \dots, 0)$, with $k$ zeros.
\end{proof}

\begin{lemma}\label{lemma:simplex-evol}
    Let $\bpi \in \mathcal{L}^n$.
    Let $\triangle = \conv(\mathcal{S}^{\bpi})$.
    Then, the Euclidean volume $\EVol(\triangle) = \sqrt{n}$.
\end{lemma}
\begin{proof}
    Let $\vw_i = s^i(\bpi) - \ve$ for all $s^i(\bpi) \in \mathcal{S}^{\bpi}$, and let $\mathcal{Q}_i$ be the set $\{\vw_{n - 1}, \vw_{n - 2}, ..., \vw_{n - i}\}$.
    That is, $\mathcal{Q}_i$ is the set of the last $i$ $\vw$-vectors.
    Let $\triangle_i$ be the convex hull of $\mathcal{Q}_i$. 
    We show $\EVol(\triangle_i) = \sqrt{i}$ for $i \geq 2$ by induction.
    
    For our base case, $i = 2.$
    We have $s^{n - 1}(\bpi) = (1, 2, \dots, n)$, so $w_{n - 1} = (0, 0, \dots, 0)$.
    By Lemma \ref{lemma:end-zeros}, we have $\vw_{n - 2} = (1, -1, 0, 0, ..., 0)$.
    So, \[\EVol(\triangle_2) = ||\vw_{n - 2} - \vw_{n - 1}|| = \sqrt{2}.\]

    For our inductive hypothesis, assume $\EVol(\triangle_i) = \sqrt{i}$. 
    We want to show $\EVol(\triangle_{i + 1}) = \sqrt{i + 1}$.

    By Lemma \ref{lemma:end-zeros}, the vectors in $\mathcal{Q}_i$ have $x_{i + 1} = x_{i + 2} = \dots = x_n = 0$.
    So, these vectors lie on the hyperplane $x_1 + x_2 + \dots + x_i = 0$, with $x_{i + 1} = \cdots = x_n = 0.$
    By Theorem \ref{thm:all-simplices}, $\triangle_i$ is a simplex, so the $i$ vectors in $\mathcal{Q}_i$ span this $(i - 1)$-dimensional hyperplane. 
    Hence, \[\EVol(\triangle_{i + 1}) = \frac{\EVol(\triangle_i)h}{i},\] where $h$ is the height of the altitude from $\triangle_i$ to $\vw_{n - (i + 1)}$.
    We show that $h = \sqrt{i(i + 1)}$.

    By Lemma $\ref{lemma:end-zeros}$, $\vw_{n - (i + 1)}$ has $x_{i + 1} = -i$.
    So, the distance from $\vw_{n - (i + 1)}$ to the hyperplane $x_{i + 1} = 0$ is $i$.
    Call the projection of $\vw_{n - (i + 1)}$ onto this hyperplane $\vv$.
    
    Now, we compute the distance from $\vv$ to the hyperplane $x_1 + x_2 + \dots + x_i = 0$.
    This hyperplane has unit normal vector $\vn = \frac{1}{\sqrt{i}}(1, 1, \dots, 1, 0, 0, \dots, 0)$ (with $i$ ones and $n - i$ zeros).
    So, the distance from $\vv$ to this hyperplane is $\vn \cdot \vv$.
    Note that $(\sqrt{i} \cdot \vn) \cdot \vv$ equals the sum of the first $i$ entries in $\vv$.
    
    We know the point $\vw_{n - (i + 1)}$ lies on the hyperplane \[x_{i + 1} = -(x_1 + x_2 + \dots + x_i).\]
    Also, by Lemma \ref{lemma:end-zeros}, $\vw_{n - (i + 1)}$ has $x_{i + 1} = -i$.
    So, the sum of the first $i$ entries of $\vv$ is $i$.
    Therefore, the distance from $\vv$ to the hyperplane $x_1 + x_2 + \dots + x_i = 0$ is $\sqrt{i}$.
    
    These two distances are perpendicular.
    So, we compute $h$:
    \[h = \sqrt{\sqrt{i}^2 + i^2} = \sqrt{i(i + 1)}.\]
    
    \noindent Therefore, \[\EVol(\triangle_{i + 1}) = \frac{\EVol(\triangle_i)h}{i} = \frac{\sqrt{i}\sqrt{i(i + 1)}}{i} = \sqrt{i + 1}.\]
    
    \noindent So, $\EVol(\triangle) = \EVol(\triangle_n) = \sqrt{n}$.
\end{proof}

We illustrate Lemma \ref{lemma:simplex-evol} with an example.

\begin{example} \label{ex:231-volume}
    Let $\bpi = 231$.
    We have $\mathcal{S}^{\bpi} = \{(2, 3, 1), (2, 1, 3), (1, 2, 3)\}$.
    So, \[\vw_0 = (1, 1, -2), \vw_1 = (1, -1, 0), \text{and } \vw_2 = (0, 0, 0).\]
    
    We have $\mathcal{Q}_2 = \{(0, 0, 0), (1, -1, 0)\}$, so $\triangle_2$ is a polytope with Euclidean volume $\sqrt{2}$.

    We also have $\mathcal{Q}_3 = \{(0, 0, 0), (1, -1, 0), (1, 1, -2)\}$.
    The point $(1, 1, -2)$ lies a distance of $2$ from the plane $x_3 = 0$. 
    The projection $(1, 1, 0)$ has a distance of $\left(\frac{1}{\sqrt{2}}(1, 1,0)\right) \cdot (1, 1,0) = \sqrt{2}$ from the line $x_1 + x_2 = 0$.
    So, the point $(1, 1, -2)$ lies a distance of $\sqrt{6}$ from conv$(\mathcal{Q}_2)$, and the volume of $\triangle_3$ is \[\frac{\sqrt{2}\sqrt{6}}{2} = \sqrt{3}.\]
    See Figure \ref{fig:euclidean_example} for a visualization of $\triangle_2$ and $\triangle_3$.
\end{example}

\begin{figure}[!ht]\label{fig:euclidean_example}

\begin{tikzpicture}
    [scale=1.2,
    vertex/.style={inner sep=1pt,circle,draw=andrespink,fill=andrespink,thick}]

\coordinate (A) at (1,1,-2); 
\node[anchor=west] at (1,1,-2) {\color{blue}$11(-2)$};
\coordinate (B) at (1,-1,0); 
\node[anchor=north] at (1,-1, 0) {\color{blue}$1(-1)0$};
\coordinate (C) at (0,0,0); 
\node[anchor=east] at (0,0,0) {\color{blue}$000$};

\coordinate (D) at  (-1,-1,0); 
\node[anchor=north] at (-1,-1,0) {\color{blue}$1(-1)0$};
\coordinate (E) at (-2,0,0); 
\node[anchor=east] at (-2,0,0) {\color{blue}$000$};

\node[anchor=east] at (-1.5,-.6,0) {$\sqrt{2}$};
\node[anchor=east] at (.5,-.6,0) {$\sqrt{2}$};
\node[anchor=east] at (.5,.6,-1) {$\sqrt{6}$};

\draw[black,top color=andresblue, bottom color=andresblue, fill opacity=0.3]  (C) -- (A) -- (B) -- cycle;
\draw[black,top color=andresblue, bottom color=andresblue, fill opacity=0.3] (D) -- (E);

\node[vertex] at (1,1,-2) {};
\node[vertex] at (1,-1,0) {};
\node[vertex] at (0,0,0) {};
\node[vertex] at (-2,0,0) {};
\node[vertex] at (-1,-1,0) {};

\end{tikzpicture}

\caption{$\triangle_2$ and $\triangle_3$ for $\bpi = 231$, as described in Example \ref{ex:231-volume}. The length of $\triangle_2$ is $\sqrt{2}$ and the area of $\triangle_3$ is $\sqrt{3}$.}
\end{figure}
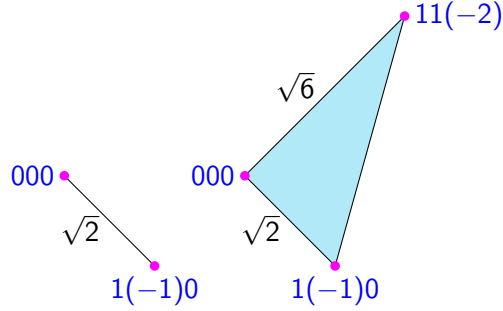

\begin{theorem}\label{thm:relvol-1}
    Let $\bpi \in \mathcal{L}^n$.
    The relative volume of $\triangle = \conv(\mathcal{S}^{\bpi})$ is 1.
\end{theorem}
\begin{proof}
    We know vol$(\triangle) = \frac{\EVol(\triangle)}{\EVol(\parallelogram)}$, where $\parallelogram$ represents the fundamental parallelepiped for $\mathrm{span}(\mathcal{S}^{\bpi})$.
    From Lemmas \ref{lemma:parallelepiped-vol} and \ref{lemma:simplex-evol}, we have \[ \EVol(\parallelogram) = \EVol(\triangle) = \sqrt{n}.\]
    So, \[\vol(\triangle) = \frac{\sqrt{n}}{\sqrt{n}} = 1.\]
\end{proof}

\section{Lattice-Point Enumeration of Stack-Sorting Simplices} \label{sec:lattice}

In this section, we focus on the lattice points of stack-sorting simplices.

\subsection{Ehrhart function and series} \label{subsec:ehrhart-background} \text{}

We give a brief introduction to Ehrhart theory, the study of lattice-point enumeration of dilates of polytopes. 
The $t^{\text{th}}$ dilate $t\mathcal{P}$ of polytope $\mathcal{P}$ is the set $\{t\vx : \vx \in \mathcal{P}\}$.
The \defterm{lattice-point enumerator} $L(\mathcal{P},t)$ of $\mathcal{P}$ counts the number of lattice points in the $t^{\text{th}}$ dilate of $\mathcal{P}$:\[L(\mathcal{P};t): = |t\mathcal{P} \cap\Z^n|.\]
If $\mathcal{P}$ is a lattice polytope, this function is a polynomial in $t$, with degree equal to the dimension of $\mathcal{P}$.
We can use the Ehrhart polynomial of $\mathcal{P}$ to obtain the \defterm{Ehrhart series} $\Ehr(\mathcal{P};z)$:
   \[\Ehr(\mathcal{P};z) := 1 + \sum_{t \in \Z_{>0}} L(\mathcal{P};t) z^t = \frac{h^{*}(\mathcal{P};z)}{(1-z)^{\dim(P)+1}}.\]
The \defterm{$h^*$-polynomial}, $h^*(\mathcal{P};z) = 1 + h^*_1 z + \cdots + h^*_{\dim(\mathcal{P})} z^{\dim(\mathcal{P})}$, has nonnegative integer coefficients.
For more information on Ehrhart theory, see \cite{ccd}.

\subsection{Hollow stack-sorting simplices} \label{subsec:2Ln1-hollow} \text{}

In \cite{lee2025stacksortingsimplicesgeometrylatticepoint}, the authors show that permutations of form $234\cdots n1$ generate hollow stack-sorting simplices.
In this section, we show that any permutation of the form $2Ln1$ generates a hollow stack-sorting simplex.
A polytope $\mathcal{P}$ is \defterm{hollow} if $\mathcal{P}$ contains no interior points.
Note that a hollow polytope can still contain boundary points.

\begin{lemma} \label{lemma:2Ln1-begins-2}
Let $\bpi \in \sym{n}$ be of the form $2Ln1$, where $L$ is any permutation of $\{3, 4, ..., (n-1)\}$.
Then, every point $s^i(\bpi) \in \mathcal{S}^{\bpi}$, besides $\mathbf{e}$, begins with 2.
\end{lemma}
\begin{proof}
    Consider the stack-sorting algorithm on $\bpi$.
    First, $2$ is pushed onto the stack.
    Then, if the next value in the permutation is not $1$, $2$ will be popped from the stack.

    Also, by Theorem 3.5 of \cite{lee2025stacksortingsimplicesgeometrylatticepoint}, we know that $s^k(\bpi)$ ends with $(n - k)1(n - k  + 1)(n - k + 2) \dots n$ for $k < n - 1$.
    So, for $k < n - 2$, the value $1$ is not one of the first two elements of $s^k(\bpi)$.
    So, for $k < n - 1$, $s^k(\bpi)$ begins with $2$.
\end{proof}

\begin{proposition} \label{prop:2Ln1-hollow}
Let $\bpi \in \sym{n}$ be of the form $2Ln1$, where $L$ is any permutation of $\{3, 4, ..., (n-1)\}$.
The simplex $\triangle := conv(\mathcal{S}^{\bpi})$ is hollow.
In particular, any non-vertex lattice point of $\triangle$ lies on the facet formed by the convex hull of $\mathcal{S}^{\bpi} \setminus \{\mathbf{e}\}$.
\end{proposition}

\begin{proof}
    Call the vertices of $\triangle$ $\textbf{v}_1, \textbf{v}_2, ..., \textbf{v}_n$, with $\textbf{v}_n = \mathbf{e}$.
    Then, by the vertex description of $\triangle$, we have
    $$\triangle = \{\lambda_1 \textbf{v}_1 + \lambda_2 \textbf{v}_2 + \cdots + \lambda_n \textbf{v}_n : \text{all } \lambda_k \geq 0 \text{ and } \lambda_1 + \lambda_2 + \cdots + \lambda_n = 1\}.$$
    We want to show that there are no lattice points on $\triangle$ with $0 < \lambda_n < 1$.
        
    Suppose the contrary that there exists a lattice point $\bp \in \mathbb{R}^n$ on $\triangle$ with $0 < \lambda_n < 1$.
    Now, consider the first value in $\bp$, $\bp_1$.
    By Lemma \ref{lemma:2Ln1-begins-2}, the first value of $\textbf{v}_k$ is 2 for $1 \leq k < n$.
    Also, we know the first value of $\vv_n = \be$ is 1.
    So,
    \[\bp_1 = 2 \lambda_1 + 2 \lambda_2 + ... + 2 \lambda_{n - 1} + 1 \lambda_n = 2(1 - \lambda_n) + \lambda_n = 2 - \lambda_n.\]
    However, $0 < \lambda_n < 1$, so, $1 < 2 - \lambda_n < 2$.
    This tells us that $\bp$ cannot be a lattice point, giving us a contradiction.
    Therefore, there are no lattice points on $\triangle$ with $0 < \lambda_n < 1$.
\end{proof}

\begin{proposition} \label{prop:2Ln1-hstar}
Let $\bpi \in \sym{n}$ be of the form $2Ln1$, where $L$ is any permutation of $\{3, 4, ..., (n-1)\}$.
Let $\triangle_1 := conv(\mathcal{S}^{\bpi})$ and $\triangle_2 := conv(\mathcal{S}^{\bpi} \setminus \{\mathbf{e}\})$.
Then, $\triangle_1$ and $\triangle_2$ have the same $h$*-polynomial.
\end{proposition}

\begin{proof}
    Call the vertices of $\triangle_1$ $\textbf{v}_1, \textbf{v}_2, ..., \textbf{v}_n$, with $\textbf{v}_n = \be$.
    Then, the vertices of $\triangle_2$ are $\textbf{v}_1, \textbf{v}_2, ..., \textbf{v}_{n - 1}$.
        
    By Corollary 3.11 in \cite{ccd}, given a vertex set $\{\textbf{v}_1, \textbf{v}_2, ..., \textbf{v}_n\}$, $h$*$_k$ equals the number of integer points in 
    \begin{equation} \label{eqn:h-integer-points}
        \{\lambda_1 \textbf{v}_1 + \lambda_2 \textbf{v}_2 + ... + \lambda_n \textbf{v}_n : \text{all } 0 \leq \lambda_i < 1 \text{ and } \lambda_1 + \lambda_2 + ... + \lambda_n = k\}.
    \end{equation}

    We want to show that all integer points on $\triangle_1$ have $\lambda_n = 0$.
    This proof follows similarly to the proof of Proposition \ref{prop:2Ln1-hollow}.
    
    Suppose the contrary that there exists a lattice point $\bp \in \mathbb{R}^n$ that satisfies (\ref{eqn:h-integer-points}) with arbitrary $k$ and $0 < \lambda_n < 1$.
    Now, consider the first value in $\bp$, $\bp_1$.
    By Lemma \ref{lemma:2Ln1-begins-2}, the first value of $\textbf{v}_k$ is 2 for $1 \leq k < n$.
    Also, we know the first value of $\vv_n = \be$ is 1.
    So, \[\bp_1 = 2\lambda_1 + 2\lambda_2 + ... + 2\lambda_{n - 1} + 1\lambda_n = 2(k - \lambda_n) + \lambda_n = 2k - \lambda_n.\]
    However, $0 < \lambda_n < 1$, so $2k - 1 < 2k - \lambda_n < 2k$.
    So, $\bp$ cannot be an integral point, and we have a contradiction.
    
    Therefore, for arbitrary $k$, an integral point satisfies (\ref{eqn:h-integer-points}) for $\triangle_1$ if and only if that point satisfies (\ref{eqn:h-integer-points}) for $\triangle_2$.
    So, $\triangle_1$ and $\triangle_2$ have the same $h$*-polynomial.
\end{proof}

\begin{remark}
We can view $\triangle_1$ as a pyramid over $\triangle_2$.
Since $\ve$ lies a distance of $1$ from the hyperplane $x_1 = 2$, we can apply Theorem 2.4 in \cite{ccd}.
Proposition \ref{prop:2Ln1-hstar} follows as a corollary.
\end{remark}

\subsection{Counting integer points}

In this subsection, we give an upper bound for the number of integer points in a stack-sorting simplex $\mathcal{P}$ generated by an $Ln1$ permutation.
This number includes interior points, boundary points, and vertex points.

A \emph{triangulation} of a polytope $\mathcal{P}$ is a partition of $\mathcal{P}$ into simplices whose vertices are lattice points.
These simplices can overlap at a facet, but they cannot otherwise intersect.

\begin{proposition} \label{prop:int_point_max}
    Let $\bpi \in \mathfrak{S}_n$ be a permutation of the form $Ln1$, and let $\mathcal{P}$ be the stack-sorting simplex generated by $\bpi$.
    Then, $\mathcal{P}$ has a maximum of $(n - 1)! + (n - 1)$ lattice points.
\end{proposition}

\begin{proof}
    Our polytope $\mathcal{P}$ is an $(n - 1)$-polytope, and the smallest $(n - 1)$-simplices have relative volume $\frac{1}{(n - 1)!}$.
    Also, by Theorem \ref{thm:relvol-1}, $\mathcal{P}$ has relative volume $1$.
    So, any triangulation of $\mathcal{P}$ would have at most $(n - 1)!$ simplices.

    Now, for the sake of contradiction, assume $\mathcal{P}$ has at least $(n - 1)! + n$ lattice points.
    Exactly $n$ of these lattice points will be vertex points.
    We first triangulate just this vertex set.
    Then, we add the lattice points to our triangulation, one at a time.

    Each new lattice point would lie in at least one simplex in our triangulation.
    Then, we can triangulate this simplex into two or more simplices using this lattice point.
    So, each new lattice point would add at least 1 simplex to our triangulation of $\mathcal{P}$.

    So, with $(n - 1)!$ non-vertex lattice points, our polytope $\mathcal{P}$ could be triangulated using at least $(n - 1)! + 1$ simplices.
    However, we have already shown that any triangulation of $\mathcal{P}$ must have at most $(n - 1)!$ simplices.

    So, we have a contradiction, and $\mathcal{P}$ must have at most $(n - 1)! + (n - 1)$ lattice points.
\end{proof}

For small values of $n$, there are stack-sorting simplices with close to $(n - 1)! + (n - 1)$ lattice points, but as $n$ grows, this approximation becomes very weak.
A stronger upper bound may lie closer to $2^{n - 1}$.
The permutation $234586791$ generates the smallest stack-sorting simplex with greater than $2^{n - 1}$ lattice points.

\begin{remark}
Readers may conjecture that all stack-sorting $d$-polytopes have a unimodular triangulation (i.e., a triangulation where all simplices have relative volume $\frac{1}{d!}$).
Unfortunately, the permutation $34251$ generates a stack-sorting simplex without a unimodular triangulation.
So, we cannot generate a lower bound for the number of lattice points in a polytope $\mathcal{P}$ in the same way we give an upper bound.
\end{remark}

\section{Questions, Conjectures, and Future Research} \label{sec:questions}
When studying and discussing stack-sorting polytopes, we encountered many ideas that may be useful areas for further exploration.
We pose the following questions to serve as directions for future research. 

\begin{enumerate}
    \item What patterns hold regarding the relative volume of stack-sorting polytopes?

    \item For $\bpi = 12543$, $\vol(\conv(\mathcal{S}^{\bpi})) = 2$.
    Which choices of $\bpi\in \mathfrak{S}_n$ form stack-sorting simplices with relative volume greater than 1? What is the maximum volume of such a polytope?

    \item What is the maximum Euclidean volume of a stack-sorting polytope? We conjecture that this value is $\sqrt{n}$ for permutations $\bpi \in \mathfrak{S}_n$, the volume of a polytope generated by an $Ln1$ permutation.

    \item Some permutations, like $23451$, generate stack-sorting polytopes with unimodular triangulations, whereas some, like $34251$, do not.
    Which stack-sorting polytopes have unimodular triangulations?


    \item What is the maximum number of lattice points in a stack-sorting simplex?

    \item For $n \le 8,$ the permutation of the form $234 \dots n1$ generates a stack-sorting simplex containing the maximum lattice points for that value of $n$. This does not hold for $n > 8$. Is there a pattern regarding which stack-sorting simplex for a given $n$ has the highest lattice point count?

    \item Does there exist a combinatorial interpretation for the $h^{*}$-coefficients of stack-sorting simplices?
\end{enumerate}

\section*{Acknowledgments}
We would like to thank the Mathematics Department at Harvey Mudd College for providing such a supportive environment to study math. 
We are grateful to Dagan Karp for feedback on a presentation of our results.
We thank DruAnn Thomas and Michael Orrison for their support throughout the research process.
We would finally like to thank our other research group members, Tito Augusto Cuchilla, Joseph Hound, Cole Plepel, and Louis Ye, for their valuable feedback on our work.


\bibliographystyle{amsplain}
\bibliography{references-new}

\end{document}